\documentclass[oneside,english]{amsart}
\usepackage[T1]{fontenc}
\usepackage[latin9]{inputenc}
\usepackage{geometry}
\usepackage{forest}
\usepackage{amstext}
\usepackage{amsthm}
\usepackage{wasysym}
\usepackage[all]{xy}

\makeatletter
\numberwithin{equation}{section}
\numberwithin{figure}{section}
\theoremstyle{plain}
\newtheorem{thm}{\protect\theoremname}
  \theoremstyle{definition}
  \newtheorem{example}[thm]{\protect\examplename}
  \theoremstyle{plain}
  
  \theoremstyle{remark}
  \newtheorem{rem}[thm]{\protect\remarkname}
\@ifpackagelater{forest}{2016/02/20}{%
\useforestlibrary*{linguistics}
}{}

\makeatother

\usepackage{babel}
  \providecommand{\corollaryname}{Corollary}
  \providecommand{\examplename}{Example}
  \providecommand{\remarkname}{Remark}
\providecommand{\theoremname}{Theorem}

\begin{document}

\title{Woon's tree and sums over compositions}

\author{C. Vignat and T. Wakhare}
\address{T. Wakhare, University of Maryland, College Park, MD 20742, USA, twakhare@gmail.com}
\address{C. Vignat, L.S.S. Supelec, Universit\'{e} Paris Sud-Orsay, France and Department of Mathematics, Tulane University, New Orleans, USA, cvignat@tulane.edu}
\begin{abstract}
This article studies sums over all compositions of an integer. We derive a generating function for this quantity, and apply it to several special functions, including various generalized Bernoulli numbers. We connect composition sums with a recursive tree introduced by S.G. Woon and extended by P. Fuchs under the name {\it general PI tree}, in which an output sequence $\{x_n\}$ is associated to the input sequence $\{g_n\}$ by summing over each row of the tree built from $\{g_n\}$. Our link with the notion of compositions allows to introduce a modification of Fuchs' tree that takes into account nonlinear transforms of the generating function of the input sequence. We also introduce the notion of \textit{generalized sums over compositions}, where we look at composition sums over each part of a composition.
\end{abstract}

\maketitle

\section{Introduction}
Sums over compositions are an object of study in their own right, and there is a vast literature considering sums over compositions or enumeration of restricted compositions. A \textit{composition} of an integer number $n$ is any sequence of
integers $n_{i}\ge1$, called \textit{parts} of $n$, such that
\[
n=n_{1}+\dots+n_{m}.
\]
There are $2$ compositions of $2:$
\[
2,\thinspace\thinspace1+1,
\]
$4$ compositions of $3:$
\[
3,\thinspace\thinspace2+1,\thinspace\thinspace1+2,\thinspace\thinspace1+1+1
\]
and more generally $2^{n-1}$ compositions of $n.$ These compositions
can be represented as follows:
\[
2:\thinspace\thinspace\newmoon\newmoon,\thinspace\thinspace\newmoon\vert\newmoon
\]
\[
3:\thinspace\thinspace\newmoon\newmoon\newmoon,\thinspace\thinspace\newmoon\newmoon\vert\newmoon,\thinspace\thinspace\newmoon\vert\newmoon\newmoon,\thinspace\thinspace\newmoon\vert\newmoon\vert\newmoon
\]
and so on. We also introduce the following notation:
let $\mathcal{C}\left(n\right)$ denote the set of all compositions
of $n;$  for an element $\pi=\left\{ n_{1},\dots,n_{m}\right\} \in \mathcal{C}\left(n\right),$
denote $m=\vert\pi\vert$ its length, i.e the number of its parts. For a sequence $\{g_n\},$ we also use the multi-index notation
\[
g_{\pi}=g_{n_{1}}\dots g_{n_{m}}.
\]
We can now state our main result, a generating function for sums over compositions:
\begin{thm}
Let $g(z) = \sum_{n \geq 1} g_n z^n $ and $f(z) = \sum_{n\geq 0 }f_n z^n$. We then have the generating function identity

\begin{equation}
f(g(z)) = f_0 + \sum_{n \geq 1} z^n \sum_{\pi \in \mathcal{C}_n} f_{\vert \pi \vert} g_{\pi}.
\end{equation}
\end{thm}

This general result unites many previous results spread across the literature, because the composite generating function $f(g(z))$ often has a simple closed form. 
%For the first time, this also introduces the nonlinear weights $f_n$. 
We link this formula to Woon's and Fuchs' trees, which provide graphical ways of visualizing it, as described below. 

In Section \ref{sec1} we introduce Woon's and Fuchs' tree, and in Section \ref{sec2} we connect them to compositions for the first time. In Section \ref{sec3}, we introduce our main results, and use them to derive a variety of closed form expressions for sums over compositions. In Section \ref{sec4} we introduce a new notation for generalized composition sums, and explore some basic identities for them.

Throughout, we apply our results to special functions and generalized Bernoulli numbers, which allows us to derive expressions for them as sums over compositions.

\section{Background}\label{sec1}

Woon's tree, as introduced by S.G. Woon in \cite{Woon}, is the following construction
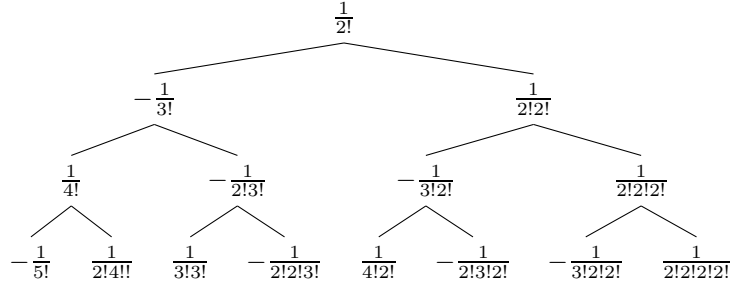
\begin{center}
\begin{forest}
[$\frac{1}{2!}$[$-\frac{1}{3!}$[$\frac{1}{4!}$[$-\frac{1}{5!}$][$\frac{1}{2!4!!}$]][$-\frac{1}{2!3!}$[$\frac{1}{3!3!}$][$-\frac{1}{2!2!3!}$]]][$\frac{1}{2!2!}$[$-\frac{1}{3!2!}$[$\frac{1}{4!2!}$][$-\frac{1}{2!3!2!}$]][$\frac{1}{2!2!2!}$[$-\frac{1}{3!2!2!}$][$\frac{1}{2!2!2!2!}$]]]]
\end{forest}
\par\end{center}

with the rule
\begin{center}
\begin{forest}
[$\pm\frac{1}{a_{1}!\dots a_{k}!}$[$\mp\frac{1}{\left(a_{1}+1\right)!\dots a_{k}!}$][$\pm\frac{1}{2!a_{1}!\dots a_{k}!}$]]
\end{forest}
\par\end{center}

\begin{flushleft}
S.G. Woon proved, using Euler MacLaurin summation formula, that the successive row sums
\begin{align*}
\frac{1}{2!} & =\frac{1}{2},\thinspace\thinspace\frac{1}{2!2!}-\frac{1}{3!}=\frac{1}{4}-\frac{1}{6}=\frac{1}{12},\thinspace\thinspace\frac{1}{4!}-\frac{1}{2!3!}-\frac{1}{3!2!}+\frac{1}{2!2!2!}=0,\\
\frac{1}{2!2!2!2!} & -\frac{3}{2!2!3!}+\frac{2}{2!4!}+\frac{1}{3!3!}-\frac{1}{5!}=-\frac{1}{720},\dots
\end{align*}
coincide with the sequence $\left\{ \left(-1\right)^{n}\frac{B_{n}}{n!}\right\} $
where the Bernoulli numbers are defined by the generating function
\[
\sum_{n\ge0}\frac{B_{n}}{n!}z^{n}=\frac{z}{e^{z}-1}.
\]
\par\end{flushleft}

Later on, Fuchs \cite{Fuchs} considered  a more general construction, the {\it general PI tree}. This associates, to the sequence of real numbers $\left\{ g_{n}\right\} ,$ which we call the {\it input sequence},
the tree
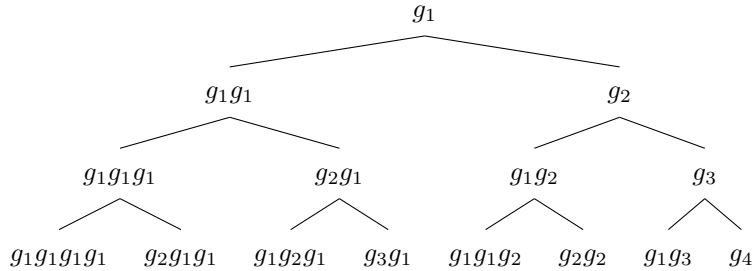
\begin{center}
\begin{forest}
[$g_{1}$[$g_{1}g_{1}$[$g_{1}g_{1}g_{1}$[$g_{1}g_{1}g_{1}g_{1}$][$g_{2}g_{1}g_{1}$]][$g_{2}g_{1}$[$g_{1}g_{2}g_{1}$][$g_{3}g_{1}$]]][$g_{2}$[$g_{1}g_{2}$[$g_{1}g_{1}g_{2}$][$g_{2}g_{2}$]][$g_{3}$[$g_{1}g_{3}$][$g_{4}$]]]]
\end{forest}
\par\end{center}

built using the two operators $P$ (for "put a 1") and $I$ (for "increase") as follows
\begin{center}
\begin{forest}
[$g_{i_{1}}g_{i_{2}}\dots g_{i_{k}}$[$g_{1}g_{i_{1}}g_{i_{2}}\dots g_{i_{k}}$, 
edge label={node[midway,left,font=\scriptsize]{$P$}}
][$g_{i_{1}+1}g_{i_{2}\dots}g_{i_{k}}$, 
edge label={node[midway,right,font=\scriptsize]{$I$}}]]
\end{forest}
\par\end{center}
The row sums of the tree generate what will be called the {\it output sequence} $\{x_n\}.$
Note that although the $g_{n}$ are real numbers, they should
be considered as noncommutating variables in the process of construction
of the tree. Moreover, Woon's tree corresponds to a PI tree with the particular choice
\[
g_{n}=\frac{-1}{\left(n+1\right)!}.
\]

Fuchs proved that if $x_{0}=1,$ then the sequence of row sums $\left\{ x_{n}\right\} _{n\ge1}$
of the general PI tree is related to the sequence of its entries $\left\{ g_{n}\right\} _{n\ge1}$
by the convolution
\begin{equation}
x_{n}=\sum_{j=1}^{n}g_{j}x_{n-j}=g_{n}+g_{n-1}x_{1}+\dots+g_{1}x_{n-1}.\label{eq:recursion}
\end{equation}

This result translates in terms of generating functions as follows - see \cite{Fuchs}:
if $\left\{ g_{n}\right\} $ and $\left\{ x_{n}\right\} $ are the
input and output sequences of Fuchs' tree, then their generating functions
\[
x\left(z\right)=\sum_{n\ge1}x_{n}z^{n},\thinspace\thinspace g\left(z\right)=\sum_{n\ge1}g_{n}z^{n}
\]
are related as
\begin{equation}
x\left(z\right)=\frac{g\left(z\right)}{1-g\left(z\right)},
\end{equation}
and
\begin{equation}
\label{CX}
g\left(z\right)=\frac{x\left(z\right)}{1+x\left(z\right)}.
\end{equation}

\begin{example}
\label{ex:Bernoulli}
In the case of Bernoulli numbers, the generating function of the input sequence is 
\[
g\left(z\right)=\sum_{k\ge1}\frac{-1}{\left(k+1\right)!}z^{k}=-\frac{1}{z}\left(e^{z}-1-z\right)=1-\frac{e^{z}-1}{z}
\]
so that the generating function of the row sums sequence is
\[
x\left(z\right)=\frac{1-\frac{e^{z}-1}{z}}{\frac{e^{z}-1}{z}}=\frac{z}{e^{z}-1}-1=\sum_{n\ge1}\frac{B_{n}}{n!}z^{n}.
\]
The recursive identity (\ref{eq:recursion}) is the well-known identity
for Bernoulli numbers \cite[24.5.3]{NIST}
\[
B_{n}=-\sum_{k=1}^{n}\binom{n}{k}\frac{B_{n-k}}{k+1}.
\]
\end{example}
\begin{example}
The Bernoulli polynomials $B_{n}\left(x\right)$ are defined by the generating function
\[
\sum_{n\ge0}\frac{B_{n}\left(x\right)}{n!}z^{n}=\frac{ze^{zx}}{e^{z}-1}.
\]
We deduce that the tree with polynomial entries
\[
g_{n}=\frac{\left(-1\right)^{n+1}}{\left(n+1\right)!}\left[x^{n+1}-\left(x-1\right)^{n+1}\right]
\]
has row sums that coincide with the Bernoulli polynomials:
\[
x_{n}=\frac{B_{n}\left(x\right)}{n!}.
\]
This gives a decomposition of Bernoulli polynomials as a sum of elementary
polynomials; for example
\[
B_{1}\left(x\right)=\frac{1}{2}\left[x^{2}-\left(x-1\right)^{2}\right]
\]
\[
\frac{B_{2}\left(x\right)}{2!}=\frac{1}{2}\left[x^{2}-\left(x-1\right)^{2}\right]-\frac{1}{6}\left[x^{3}-\left(x-1\right)^{3}\right].
\]
The general expression for $\frac{B_{n}\left(x\right)}{n!}$ is given
in Thm. \ref{Bernoulli expansion} below.
\end{example}

\begin{example}
Higher-order Bernoulli numbers, also called N\"{o}rlund polynomials, are defined for an integer parameter
$p$ by the generating function
\begin{equation}
\label{Norlund gf}
\sum_{n\ge0}\frac{B_{n}^{\left(p\right)}}{n!}z^{n}=\left(\frac{z}{e^{z}-1}\right)^{p}.
\end{equation}

The output sequence 
\[
x_{n}=\frac{B_{n}^{\left(p\right)}}{n!}
\]
corresponds to the input sequence
\[
g_{n}=-\frac{p!}{\left(p+n\right)!}\left\{ \begin{array}{c}
n+p\\
p
\end{array}\right\} 
\]
where $\left\{ \begin{array}{c}
n\\
p
\end{array}\right\} $ is the Stirling number of the second kind. 
%We also have
%\[
%g_{n}=-\sum_{k_{1}+\dots+k_{p}=n}\frac{1}{\left(k_{1}+1\right)!\dots\left(k_{p}+1\right)!}.
%\]
\begin{example}
The hypergeometric Bernoulli numbers, as introduced by F.T. Howard \cite{Howard}, are defined, for $a>0$ and $b>0,$ by the generating function
\[
\sum_{n\ge0}\frac{B_{n}^{\left(a,b\right)}}{n!}z^{n}=\frac{1}{_{1}F_{1}\left(\begin{array}{c}
a\\
a+b
\end{array};z\right)}
\]
where $_{1}F_{1}$ is the hypergeometric function
\[
_{1}F_{1}\left(\begin{array}{c}
a\\
c
\end{array};z\right) = \sum_{n\ge0} \frac{\left(a\right)_{n}}{\left(c\right)_{n}} \frac{z^n}{n!}
\]
with the notation $\left(a\right)_{n}$ for the Pochhammer symbol
\[
\left(a\right)_{n}=\frac{\Gamma\left(a+n\right)}{\Gamma\left(a\right)}.
\]
The output sequence
\[
x_{n}=\frac{B_{n}^{\left(a,b\right)}}{n!}
\]
corresponds to the input sequence
\[
g_{n}=-\frac{1}{n!}\frac{\left(a\right)_{n}}{\left(a+b\right)_{n}}.
\]
\end{example}
\end{example}

\section{Connection to compositions}\label{sec2}

\subsection{Definitions}

In the examples studied so far, we were able to compute a few row
sums of Fuchs' tree, but we still need a general and non-recursive formula that gives
the row sum sequence $\left\{ x_{n}\right\} $ explicitly as a function
of the input sequence $\left\{ g_{n}\right\} .$ This implies
a better description of the row generating process of Woon's tree,
which requires the notion of compositions.

The representation
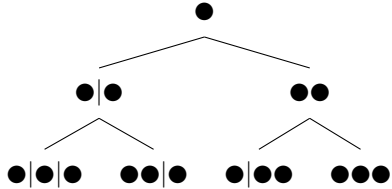
\begin{center}
\begin{forest}
[$\newmoon$[$\newmoon$$\vert$$\newmoon$[$\newmoon$$\vert$$\newmoon$$\vert$$\newmoon$][$\newmoon$$\newmoon$$\vert$$\newmoon$]][$\newmoon$$\newmoon$[$\newmoon$$\vert$$\newmoon$$\newmoon$][$\newmoon$$\newmoon$$\newmoon$]]]
\end{forest}
\par\end{center}
suggests a natural bijection between the generation process of the next row
in Woon's tree and the generation process of the compositions of the integer
$n+1$ in terms of those of $n.$ We restate Fuchs' result \cite{Fuchs} in terms of compositions, 
and then apply this result to derive several new identities for Catalan numbers and Hermite polynomials.

\subsection{Fuchs' result}
We can now restate Fuchs' main result \eqref{eq:recursion} as follows
%, which corresponds with our Theorem \ref{thm:main} with $f_n=1,\,\,n\ge 1$:
\begin{thm}\label{thm:5}
The sequence of row sums $\left\{ x_{n}\right\} $ in Woon's tree
can be computed from its sequence of entries $\left\{ g_{n}\right\} $
as the sum over compositions
\begin{equation}
x_{n}=\sum_{\pi\in\mathcal{C}\left(n\right)}g_{\pi}=\sum_{p=1}^n\sum_{\underset{k_{i}\ge1}{k_{1}+\dots+k_{p}=n}}g_{k_{1}}\dots g_{k_{p}}.\label{eq:compositions}
\end{equation}
This sum over compositions can also be expressed
as the weighted sum over convolutions
\begin{equation}
x_{n}=\sum_{p=1}^{n}\binom{n+1}{p+1}\sum_{\underset{k_{i}\ge 0}{k_{1}+\dots+k_{p}=n}}g_{k_{1}}\dots g_{k_{p}}.\label{eq:convolutions}
\end{equation}
Moreover, these relations can be inverted by exchanging $x_n$ with $-g_n$
: the input sequence $\{g_n\}$ in Woon's tree can be recovered from its row sum sequence $\{x_n\}$ as the sum over compositions
\begin{equation}
g_{n} =\sum_{\pi\in\mathcal{C}\left(n\right)}\left(-1\right)^{\pi+1}x_{\pi}=\sum_{p=1}^n\left(-1\right)^{p+1}\sum_{\underset{k_{i}\ge1}{k_{1}+\dots+k_{p}=n}}x_{k_{1}}\dots x_{k_{p}}
\label{eq:inverse compositions}
\end{equation}
or as the weighted sum over convolutions
\begin{equation}
x_{n}=\sum_{p=1}^{n}\left(-1\right)^{p+1}\binom{n+1}{p+1}\sum_{\underset{k_{i}\ge 0}{k_{1}+\dots+k_{p}=n}}x_{k_{1}}\dots x_{k_{p}}.
\end{equation}
\end{thm}
\begin{proof}
Denote $g_{n}=-\frac{w_{n}}{n!}$ so that
\begin{align*}
x_{n} & =\sum_{m=1}^{n} \sum_{\underset{k_{i}\ge1}{k_{1}+\dots+k_{m}=n}}g_{k_{1}}\dots g_{k_{m}}=\sum_{m=1}^{n}\left(-1\right)^{m}\sum_{\underset{k_{i}\ge1}{k_{1}+\dots+k_{m}=n}}\frac{w_{k_{1}}}{k_{1}!}\dots\frac{w_{k_{m}}}{k_{m}!}\\
 & =\sum_{m=1}^{n} \frac{\left(-1\right)^{m}}{n!} \sum_{\underset{k_{i}\ge1}{k_{1}+\dots+k_{m}=n}}\binom{n}{k_{1},\dots,k_{m}}w_{k_{1}}\dots w_{k_{m}}.
\end{align*}
Denote the incomplete sum
\[
\tilde{S}_{m,n}=\sum_{\underset{k_{i}\ge1}{k_{1}+\dots+k_{m}=n}}\binom{n}{k_{1},\dots,k_{m}}w_{k_{1}}\dots w_{k_{m}}
\]
and its complete version
\[
S_{m,n}=\sum_{\underset{k_{i}\ge0}{k_{1}+\dots+k_{m}=n}}\binom{n}{k_{1},\dots,k_{m}}w_{k_{1}}\dots w_{k_{m}}
\]
so that
\[
x_{n}=\sum_{m=1}^{n}\frac{\left(-1\right)^{m} }{n!}\tilde{S}_{m,n}
\]
and both variables $S_{m,n}$ and $\tilde{S}_{m,n}$ are related as
\[
\tilde{S}_{m,n}=\sum_{p=1}^{m}\left(-1\right)^{m-p}\binom{m}{p}S_{p,n},
\]
that can be proved for example by induction on $m.$
We deduce, following the same steps as in the Bernoulli case above,
\[
x_{n}=\sum_{m=1}^{n}\frac{\left(-1\right)^{m} }{n!}\sum_{p=1}^{m}\left(-1\right)^{m-p}\binom{m}{p}S_{p,n}=\sum_{p=1}^{n}\frac{\left(-1\right)^{p}}{n!}S_{p,n}\sum_{m=p}^{n} \binom{m}{p}.
\]
Moreover
\[
S_{p,n}=\sum_{\underset{k_{i}\ge0}{k_{1}+\dots+k_{p}=n}}\binom{n}{k_{1},\dots,k_{p}}w_{k_{1}}\dots w_{k_{p}}=\left(-1\right)^{p}n!\sum_{\underset{k_{i}\ge0}{k_{1}+\dots+k_{p}=n}}g_{k_{1}}\dots g_{k_{p}}
\]
and the proof of the first part follows.
The inversion of these relations is deduced from the identity \eqref{CX}.\end{proof}

The result of Thm. \ref{thm:5} has been rediscovered many times in relation to sums over compositions. For examples of identities derived from this result, see \cite{Sills}, \cite{Hoggatt} and  \cite{Gessel}. Selecting various $g_{k} =1$ for $k$ in some set of indices $J$, and $g_i = 0$ otherwise, also gives the generating function for the number of compositions into parts from this set $J$. For example, we can let $g_{2n}=0$ and $g_{2n+1}=1$ to find a formula for the number of compositions into odd parts.

\subsection{Invariant sequences}
In this section, we look for sequences that are invariant by Fuchs' tree.
\subsubsection{Catalan numbers}

The Catalan numbers are defined by the generating function
\[
\sum_{n\ge0}C_{n}z^{n}=\frac{1-\sqrt{1-4z}}{2z}.
\]
They are invariants of Woon's tree in the following sense.
\begin{thm}
If 
\[
g_{n}=\begin{cases}
C_{n-1} & n\ge1\\
1 & n=0
\end{cases}
\]
is the input sequence of Woon's tree, then the row sums are
\[
x_{n}=C_{n},\thinspace\thinspace n\ge0.
\]
As a consequence, the Catalan numbers satisfy the sum over compositions
identities
\begin{equation}
C_{n}=\sum_{m=1}^{n}\sum_{\underset{k_{i}\ge1}{k_{1}+\dots+k_{m}=n}}C_{k_{1}-1}\dots C_{k_{m}-1}.\label{eq:Catalan sum over compositions}
\end{equation}
and
\begin{equation}
C_{n-1}=\sum_{m=1}^{n}\left(-1\right)^{m+1}
\sum_{\underset{k_{i}\ge1}{k_{1}+\dots+k_{m}=n}}C_{k_{1}}\dots C_{k_{m}}.
\label{eq:inverse Catalan sum over compositions}
\end{equation}

\end{thm}
\begin{rem}
Identity \eqref{eq:Catalan sum over compositions} can be considered as a generalization of the classic convolution identity for Catalan numbers \cite[26.5.3]{NIST}
\[
C_n = \sum_{k=0}^{n-1}C_{k}C_{n-1-k}.
\]
\end{rem}
\begin{proof}
From the generating function of the input sequence
\[
g\left(z\right)=\sum_{n\ge1}g_{n}z^{n}=\sum_{n\ge1}C_{n-1}z^{n}=z\frac{1-\sqrt{1-4z}}{2z}=\frac{1-\sqrt{1-4z}}{2},
\]
we deduce the generating function of the row sums
\[
X\left(z\right)=\frac{g\left(z\right)}{1-g\left(z\right)}=\frac{1-\sqrt{1-4z}}{1+\sqrt{1-4z}}=-1+\frac{1-\sqrt{1-4z}}{2z}=\sum_{n\ge1}C_{n}z^{n},
\]
which proves the result. Both identities for Catalan numbers \eqref{eq:Catalan sum over compositions} and \eqref{eq:inverse Catalan sum over compositions}
are a consequence of \eqref{eq:compositions} and \eqref{eq:inverse compositions}.
\end{proof}

\subsubsection{Hermite polynomials}

Another invariant sequence of Woon's tree is the sequence of Hermite polynomials $\left\{ H_{n}\left(x\right)\right\}$  defined by the generating function
\[
\sum_{n\ge0}\frac{H_{n}\left(x\right)}{n!}z^{n}=e^{2xz-z^{2}}.
\]
It can be checked that if the entries of Woon's tree are chosen as
\[g_{n}=-\imath^{n}\frac{H_{n}\left(\imath x\right)}{n!}
\]
then the sequence of row sums is equal to
\[
x_{n}=\frac{H_{n}\left(x\right)}{n!}.
\]

\section{A Nonlinear Generalization}\label{sec3}
We are able to generalize these results to sums over compositions with weights, which corresponds to a further generalization of Fuchs' tree. Because Fuchs' tree graphically represents sums over compositions, we can represent the generating function $f(g(z))$ as the convolution of two trees:
\vspace{0.5cm}

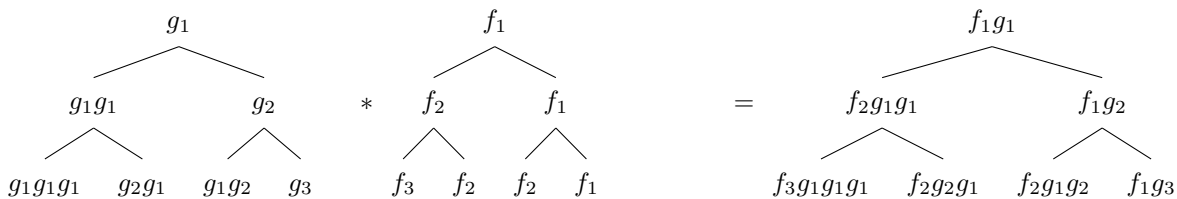
\begin{minipage}{1\textwidth}%
\noindent\begin{minipage}{0.3\textwidth}%
\begin{forest}
[$g_{1}$[$g_{1}g_{1}$[$g_{1}g_{1}g_{1}$][$g_{2}g_{1}$]][$g_{2}$[$g_{1}g_{2}$][$g_{3}$]]]
\end{forest}
\end{minipage}
$*$
\noindent\begin{minipage}{0.3\textwidth}%
%\begin{center}
\begin{forest}
[$f_{1}$[$f_{2}$[$f_{3}$][$f_{2}$]][$f_{1}$[$f_{2}$][$f_{1}$]]]\end{forest}
%\par\end{center}
\end{minipage}%
$=$
\noindent\begin{minipage}{0.3\textwidth}%
%\begin{center}
\begin{forest}
[$f_{1}g_{1}$[$f_{2}g_{1}g_{1}$[$f_{3}g_{1}g_{1}g_{1}$][$f_{2}g_{2}g_{1}$]][$f_{1}g_{2}$[$f_{2}g_{1}g_{2}$][$f_{1}g_{3}$]]]\end{forest}
%\par\end{center}
\end{minipage}%
\end{minipage}
\vspace{0.5cm}

The analytic description of this operation is as follows.
\begin{thm}\label{thm:main}
Let $g(z) = \sum_{n \geq 1} g_n z^n $ and $f(z) = \sum_{n\geq 0 }f_n z^n$. We then have the generating function identity
\begin{equation}
\label{eq:f(g(z))}
f(g(z)) = f_0 + \sum_{n \geq 1} z^n \sum_{\pi \in \mathcal{C}_n} f_{\vert \pi \vert} g_{\pi}.
\end{equation}

\end{thm}
\begin{proof}
We note that the case $f_0=0$ and $f_n = 1$, $n \geq 1$, corresponds to Theorem \ref{thm:5}. Then $f(z) = \frac{z}{1-z}$, leading to the observed relationship between $g(z)$ and $x(z)$. In terms of trees, the sum $\sum_{\pi \in \mathcal{C}_n} f_{\vert \pi \vert} g_\pi$ corresponds to a row sum, with additional weights depending on how many different parts are in each composition of $n$. With the notation $T_n(f;a) = \frac{f^{(n)}(a)}{n!}$, we have the following result of Vella \cite{Vella}, based off the classical Fa\'{a} di Bruno formula:
%\begin{align}
%T_n(f \circ g;a) &= \sum_{\pi \in \mathcal{P}_n} \binom{\vert \pi \vert}{\delta(\pi)} T_{\vert \pi \vert}(f;g(a)) \prod_{i=1}^n [T_i(g;a)]^{\pi_i} \\
%&= \sum_{\pi \in C_n}  T_{\vert \pi \vert}(f;g(a)) \prod_{i=1}^n [T_i(g;a)]^{\pi_i}.
%\end{align}
\begin{equation}
T_n(f \circ g;a) = \sum_{\pi \in C_n}  T_{\vert \pi \vert}(f;g(a)) \prod_{i=1}^n [T_i(g;a)]^{\pi_i}.
\end{equation}

We let $a=0$, so that $g(a)=0$. Then letting $T_n(f;0) =f_n$ and $T_n(g;0) = g_n$ yields the result.
\end{proof}
\begin{rem}
Eger \cite{Eger} explored the case $f_k=1$ and $f_n=0$ for $n \neq k$. He used this to define 'extended binomial coefficients', which constitute a special case of our results.
\end{rem}
\begin{rem}
One consequence of identity \eqref{eq:f(g(z))} is as follows: let $\{\mu_n\}$ denote a moment sequence and $\{\kappa_n\}$ a cumulant sequence of a random variable $Z$: the moments are the expectations $\mu_n = \mathbb{E} X^n$ and the cumulants are the Taylor coefficients of the logarithm of the moment generating function $\sum_{n \ge 0} \frac{\mu_n}{n!} z^n$. Choosing
\[
g_1\left(z\right)= \sum_{n=1}^{\infty} \mu_n \frac{z^n}{n!},\thinspace\thinspace f\left(z\right)=\log\left(1+z\right)
\]
so that
\[
f\left(g_1\left(z\right)\right)=\sum_{n=1}^{\infty}  \kappa_n \frac{z^n}{n!},
\]
gives the identity between moments and cumulants
\[
\sum_{\pi\in\mathcal{C}\left(n\right)}\frac{\left(-1\right)^{\vert \pi \vert}}{\vert \pi \vert}\frac{\mu_\pi}{\pi!}=\frac{\kappa_n}{n!}
\]
Let $g_2(z) = \sum_{n=1}^{\infty}  \kappa_n \frac{z^n}{n!}$ be the cumulant generating function and $f(z) = e^z-1$, so that $f(g_2(z))$ is the moment generating function. This gives the other identity between moments and cumulants
\[
\sum_{\pi\in\mathcal{C}\left(n\right)}\frac{1}{\vert \pi \vert!}\frac{\kappa_\pi}{\pi!}=\frac{\mu_n}{n!}.
\]
This allows us to express moments and cumulants in terms of sums over compositions of each other. Note that by definition, if $\pi = k_1 \dots k_p$ then $\left(-1\right)^{\pi}=\left(-1\right)^{k_1 +\dots+ k_p}$, $\mu_\pi = \mu_{k_1}\cdots\mu_{k_p}$, $\kappa_\pi = \kappa_{k_1}\cdots\kappa_{k_p}$, $\pi! = k_{1}! \dots k_{p}!$ and $\vert \pi \vert!=p!.$
\end{rem}

\begin{thm}
\label{non linear}
We have the following identity for sums over compositions:
\begin{equation}\label{eq:convolutions}
\sum_{\pi \in \mathcal{C}_n} f_{\vert \pi \vert} g_{\pi} = \sum_{p=1}^{n} f_p\sum_{\underset{k_{i}\ge1}{k_{1}+\dots+k_{p}=n}} g_{k_{1}}\dots g_{k_{p}} = \sum_{p=1}^{n}\left(\sum_{m=p}^{n}f_m \binom{m}{p}\right)\sum_{k_{1}+\dots+k_{p}=n}g_{k_{1}}\dots g_{k_{p}}.
\end{equation}
\end{thm}
\begin{proof}
The proof is identical to that given in Theorem \ref{thm:5}.
\end{proof}

We can also, for the first time, find a general formula for sums over all parts of all compositions of $n$. This corresponds to an 'additive' Woon tree where the creation operator "I" adds $g_1$ instead of multiplying by it.
\begin{thm}
Let $g(z) = \sum_{n \geq 1} g_n z^n $ and $f(z) = \sum_{n\geq 0 }f_n z^n$. We then have the generating function identity

\begin{equation}
f'\left(\frac{z}{1-z}\right) g(z)=  \sum_{n \geq 1} z^n \sum_{\pi \in \mathcal{C}_n} f_{\vert \pi \vert} \sum_{k_i \in \pi} g_{k_i}.
\end{equation}
\end{thm}
\begin{proof}
We begin with $g(z) = \sum_{n \geq 1} g_n z^n $ and transform it to the bivariate generating function $G(z, \lambda) =  \sum_{n \geq 1} \lambda^{g_n} z^n$. We then take a partial derivative with respect to $\lambda$ and evaluate at $\lambda=1$, so
\begin{equation}
\frac{\partial}{\partial \lambda}  f(G(z,\lambda)) \Bigr|_{\lambda=1}= \frac{\partial}{\partial \lambda}   \sum_{n \geq 1} z^n \sum_{\pi \in \mathcal{C}_n} f_{\vert \pi \vert} \lambda^{ \sum_{k_i \in \pi} g_{k_i}}  \Bigr|_{\lambda=1}.
\end{equation}
Noting that $G(z,1) =\frac{z}{1-z}$ and $\frac{\partial}{\partial \lambda}  G(z,\lambda) |_{\lambda=1} = g(z)$ yields the theorem.
\end{proof}

\begin{rem}
As a consequence of this result, choosing $f\left(z\right) = \frac{z}{1-z},$ we deduce the generating function of the sequence 
\[
\sum_{\pi \in \mathcal{C}_n}\sum_{k_i \in \pi} g_{k_i}
\]
of a sequence $\{g_n\}$
as
\[
\sum_{n \geq 1} z^n \sum_{\pi \in \mathcal{C}_n} \sum_{k_i \in \pi} g_{k_i}.
= \left(\frac{1-z}{1-2z}\right)^2 g\left(z\right).
\]
Generating functions for other sequences related to compositions, such as the total number of summands in all the compositions of $n,$ can be found for example in \cite{Merlini}.
\end{rem}

\subsection{Back to the Bernoulli numbers}

Going back to Example \ref{ex:Bernoulli} and Equation \eqref{eq:convolutions}, we deduce the expression
\[
\frac{B_{n}}{n!}=\sum_{\pi\in\mathcal{C}\left(n\right)}\frac{\left(-1\right)^{\vert \pi \vert}}{\left(\pi+1\right)!}= \sum_{p=1}^{n}\binom{n+1}{p+1}\left(-1\right)^{p}\sum_{k_{1}+\dots+k_{p}=n}\frac{1}{\left(k_{1}+1\right)!\dots\left(k_{p}+1\right)!}
\]
with the notation
\[
\left(\pi+1\right)!=\prod_{i}\left(k_{i}+1\right)!.
\]
%%Here, we have taken $f_n=0$. 
From the multinomial identity, we know that, in terms of Stirling numbers of the second kind,
\begin{equation}
\label{Stirling}
\sum_{k_{1}+\dots+k_{p}=n}\binom{n}{k_{1},\dots,k_{p}}\frac{1}{\left(k_{1}+1\right)\dots\left(k_{p}+1\right)} =\frac{n!p!}{\left(n+p\right)!}\left\{ \begin{array}{c}
n+p\\
p
\end{array}\right\} .
\end{equation}
We deduce the expression
\[
B_{n}=\sum_{p=1}^{n}\left(-1\right)^{p}\frac{\left\{ \begin{array}{c}
n+p\\
p
\end{array}\right\} }{\binom{n+p}{p}}\binom{n+1}{p+1},\,\,n \ge 1.
\]
%for the Bernoulli numbers, which can be found as \cite[24.15.7]{NIST} \footnote{Note that the sum starts at $p=0$ in \cite[24.15.7]{NIST}, but for $n \ge 1$, the Stirling coefficient $\left\{ \begin{array}{c}
%n\\
%0
%\end{array}\right\}=0$, so that both identities coincide.}. 
Another expression for the Bernoulli numbers can be obtained by choosing
\[
g\left(z\right)=e^{z}-1=\sum_{n\ge1}\frac{z^{n}}{n!}
\]
and 
\[
f\left(z\right)=\frac{\log\left(1+z\right)}{z}=\sum_{n\ge0}\frac{\left(-1\right)^{n}}{n+1}z^{n}
\]
giving
\begin{eqnarray}
\frac{B_{n}}{n!} &=&\sum_{\pi\in C\left(n\right)}\frac{\left(-1\right)^{\vert \pi \vert}}{\vert \pi \vert+1}\frac{1}{\pi!}\\
&=&\sum_{p=1}^{n}\frac{\left(-1\right)^p}{p+1}\sum_{\underset{k_{i}\ge1}{k_{1}+\dots+k_{p}=n}}\frac{1}{k_{1}!\dots k_{p}!} = \sum_{p=1}^{n}\frac{\left(-1\right)^p}{p+1}p!
\left\{ 
\begin{array}{c}
n\\p
\end{array}
\right\}
\end{eqnarray}
which can be found under a slightly different form as \cite[Entry 24.6.9]{NIST}.

We note that we have two distinct representations of the Bernoulli numbers as sums over compositions:
\[
\frac{B_{n}}{n!} =\sum_{\pi\in\mathcal{C}\left(n\right)}\frac{\left(-1\right)^{\vert \pi \vert}}{\left(\pi+1\right)!} =\sum_{\pi\in C\left(n\right)}\frac{\left(-1\right)^{\vert \pi \vert}}{\vert \pi \vert+1}\frac{1}{\pi!}.
\]
We also note that in \cite{Vella},
D.C. Vella obtains other expressions of the Bernoulli and Euler numbers
as sums over compositions.

\subsection{Back to Bernoulli polynomials}

We can now provide a general formula for the Bernoulli polynomials
as follows:
\begin{thm}
\label{Bernoulli expansion}
The Bernoulli polynomials can be expanded as
\[
B_{n}\left(x\right)=\sum_{p=1}^{n}\binom{n+1}{p+1}\left(-1\right)^{p}B_{n}^{\left(-p\right)}\left(-px\right),\thinspace\thinspace n\ge1.
\]
where $B_{n}^{\left(p\right)}\left(x\right)$ is the higher-order Bernoulli polynomial with parameter $p$ with generating function 
\begin{equation}
\label{gf generalized Bernoulli}
\sum_{n\ge0}\frac{B_{n}^{\left(p\right)}\left(x\right)}{n!}z^{n}=e^{zx}\left(\frac{z}{e^{z}-1}\right)^{p}.
\end{equation}

\end{thm}
\begin{proof}
Using
\[
g_{n}=\frac{\left(-1\right)^{n+1}}{\left(n+1\right)!}\left(x^{n+1}-\left(x-1\right)^{n+1}\right)=\frac{\left(-1\right)^{n+1}}{n!}\int_{0}^{1}\left(x+u-1\right)^{n}du,
\]
we deduce
\[
x_{n}=\frac{B_{n}\left(x\right)}{n!}=\sum_{p=1}^{n}\binom{n+1}{p+1}\sum_{k_{1}+\dots+k_{p}=n}g_{k_{1}}\dots g_{k_{p}}
\]
with
\begin{align*}
\sum_{k_{1}+\dots+k_{p}=n}g_{k_{1}}\dots g_{k_{p}} & =\frac{\left(-1\right)^{n+p}}{n!}\sum_{k_{1}+\dots+k_{p}=n}\binom{n}{k_{1},\dots,k_{p}}\int_{0}^{1}\dots\int_{0}^{1}\prod_{i=1}^{p}\left(x+u_{i}-1\right)^{k_{i}}du_{1}\dots du_{p}\\
 & =\frac{\left(-1\right)^{n+p}}{n!}\int_{0}^{1}\dots\int_{0}^{1}\left(px-p+u_{1}+\dots+u_{p}\right)^{n}du_{1}\dots du_{p}.
\end{align*}
This sequence has generating function
\begin{align*}
\sum_{n\ge0}\left(\frac{\left(-1\right)^{n+p}}{n!}\int_{0}^{1}\dots\int_{0}^{1}\left(px-p+u_{1}+\dots+u_{p}\right)^{n}du_{1}\dots du_{p}\right)z^{n} & =\left(-1\right)^{p}e^{-z\left(px-p\right)}\left(\frac{1-e^{-z}}{z}\right)^{p}\\
 & =\left(-1\right)^{p}e^{-zpx}\left(\frac{e^{z}-1}{z}\right)^{p};
\end{align*}
comparing to the generating function of the higher-order Bernoulli
polynomials \eqref{gf generalized Bernoulli},
we deduce
\[
\sum_{k_{1}+\dots+k_{p}=n}g_{k_{1}}\dots g_{k_{p}}=\frac{\left(-1\right)^{p}}{n!}B_{n}^{\left(-p\right)}\left(-px\right)
\]
and
\[
B_{n}\left(x\right)=\sum_{p=1}^{n}\binom{n+1}{p+1}\left(-1\right)^{p}B_{n}^{\left(-p\right)}\left(-px\right)
\]
which is the desired result.
\end{proof}

\subsection{Back to higher-order Bernoulli polynomials}

We apply the result of Thm. \ref{thm:main} to higher-order Bernoulli numbers as follows: starting
from
\[
g_{n}=-\frac{1}{\left(n+1\right)!},\thinspace\thinspace g\left(z\right)=1-\frac{e^{z}-1}{z},
\]
the desired generating function is
\[
f\left(g\left(z\right)\right)=\left(\frac{z}{e^{z}-1}\right)^{q}-1=\left(\frac{1}{1-g\left(z\right)}\right)^{q}-1
\]
so that
\[
f\left(z\right)=\left(\frac{1}{1-z}\right)^{p}-1=\sum_{k\ge1}\binom{k+p-1}{p-1}z^{k}.
\]
Applying formula \eqref{eq:convolutions}, we deduce
\[
\frac{B_{n}^{\left(q\right)}}{n!}=\sum_{p=1}^{n}\left(\sum_{m=p}^{n}\binom{m}{p}\binom{m+q-1}{q-1}\right)\sum_{k_{1}+\dots+k_{p}=n}\frac{\left(-1\right)^{p}}{\left(k_{1}+1\right)!\dots\left(k_{p}+1\right)!}.
\]
The inner sum is easily computed as
\[
\sum_{m=p}^{n}\binom{m}{p}\binom{m+q-1}{q-1}=\frac{n-p+1}{p+q}\binom{n+1}{p}\binom{n+q}{q-1}
\]
whereas the last sum is, by \eqref{Stirling}, equal to
\[
\sum_{k_{1}+\dots+k_{p}=n}\frac{\left(-1\right)^{p}}{\left(k_{1}+1\right)!\dots\left(k_{p}+1\right)!}=\left(-1\right)^{p}\frac{p!}{\left(n+p\right)!}\left\{ \begin{array}{c}
n+p\\
p
\end{array}\right\} 
\]
so that
\begin{align*}
B_{n}^{\left(q\right)} & =\binom{n+q}{q-1}\sum_{p=1}^{n}\left(-1\right)^{p}\frac{\left\{ \begin{array}{c}
n+p\\
p
\end{array}\right\} }{\binom{n+p}{p}}\left(\frac{n+1}{p+q}\right)\binom{n}{p}\\
 & =\sum_{p=1}^{n}\left(-1\right)^{p}\frac{\left\{ \begin{array}{c}
n+p\\
p
\end{array}\right\} }{\binom{n+p}{p}}\binom{n+q}{n-k}\binom{q+k-1}{k}.
\end{align*}
This identity can be found as \cite[Eq.(15)]{Srivastava}.

\subsection{Compositions with restricted summands}

When the input sequence $\{ g_{n} \}$ has a simple structure, more details can be obtained about the row-sum sequence $\{x_{n}\}.$ 
We start with the example of Fibonacci numbers,
by explicitly constructing Woon's tree for the Fibonacci case
\[
x_{n}=F_{n}
\]
with 
\[
F_{0}=1,\thinspace\thinspace F_{1}=1,\thinspace\thinspace F_{2}=2,\thinspace\thinspace F_{3}=3,\thinspace\thinspace F_{4}=5,\thinspace\thinspace F_{5}=8\dots
\]
Since the Fibonacci numbers satisfy the recurrence
\[
F_{n}=F_{n-1}+F_{n-2},\,\,n \ge 2,
\]
we deduce from (\ref{restricted}) that the sequence of entries
of the tree satisfies
\[
g_{k}=\begin{cases}
1 & k=1,2\\
0 & \text{else}
\end{cases}.
\]
The corresponding Woon's tree starts as 
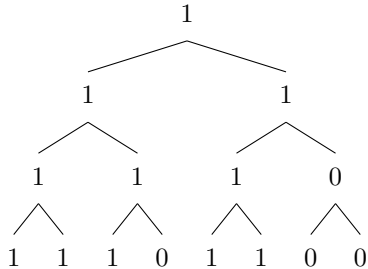
\begin{center}
\begin{forest}
[$1$[$1$[$1$[$1$][$1$]][$1$[$1$][$0$]]][$1$[$1$[$1$][$1$]][$0$[$0$][$0$]]]]
\end{forest}
\par\end{center}
\begin{flushleft}
We deduce the expression for the Fibonacci numbers as
\[F_{n}=\mathcal{C}_{n}^{\left\{ 1,2\right\}}=\#\left\{ \left(k_{1},\dots,k_{m}\right)\vert k_{1}+\dots+k_{m}=n,\thinspace\thinspace k_{i}=1,2\right\},
\] 
the number of compositions of $n$ with parts equal to $1$ or $2$, which can be found in \cite[p.42]{Flajolet} and \cite{Alladi}. Moreover, we recover the generating function $\frac{z+z^2}{1-z-z^2} = \sum_{n\geq 1} F_n z^n$.
\par\end{flushleft}

This result can be generalized remarking that we have a correspondence between linear recurrences and compositions into restricted parts. Fix a set $J$ of integers: we adopt the notation 
\[
\mathcal{C}_{n}^{\left\{J\right\} }=\#\left\{ \left(k_{1},\dots,k_{m}\right)\vert k_{1}+\dots+k_{m}=n,\thinspace\thinspace k_{i}\in J\right\}
\]
from \cite[p.42]{Flajolet} to denote the number of compositions of $n$ with parts in the set $J$.

\begin{thm}\label{restricted}
Let $J $ be a finite set of positive indices and consider the linear recurrence $x_n  =  \sum_{j \in J} x_{n-j}$. Then we have the dual identities:
\begin{equation}\label{eq:comp}
x_n  = \mathcal{C}_{n}^{\left\{J\right\} },
\end{equation}
\begin{equation}\label{eq:compgen}
\sum_{n \geq 1} x_n z^n  =\frac{\sum_{j \in J}z^j}{1-\sum_{j \in J}z^j},
\end{equation}
Furthermore, the corresponding Woon tree has input sequence
\begin{equation}\label{eq:comptree}
g_{k}=\begin{cases}
1 & k \in J\\
0 & \text{else}
\end{cases}.
\end{equation}
\end{thm}
\begin{proof}
We begin with the recurrence \eqref{eq:recursion}, $x_n = \sum_{j=1}^n g_j x_{n-j}$, where $x_n$ corresponds to the $n-$th row sum of Woon's tree. Since $x_n  =  \sum_{j\in J} x_{n-j}$, the set $J$ consists of the indices $i$ such that $g_i =1$. We also have that $g_i=0$ everywhere else. This sequence then satisfies the conditions to be considered a generalized Woon tree, so we have the realization \eqref{eq:comptree}. We can then apply the transform $x(z) = \frac{g(z)}{1-g(z)}$, which yields \eqref{eq:compgen}, since $g(z)=\sum_{n \geq 1}g_n z^n = \sum_{j \in J}z^j$. Finally, we can apply \eqref{eq:convolutions} with $f_0 =0$ and $f_n=1$ to yield 
\begin{equation*}
\frac{g(z)}{1-g(z)} =  \sum_{n \geq 1} z^n \sum_{\underset{k_{i}\ge1,\thinspace m\ge1}{k_{1}+\dots+k_{m}=n}} g_{k_{1}}\dots g_{k_{m}} =   \sum_{n \geq 1} z^n \sum_{\underset{k_{i}\ge1,\thinspace k_i \in J,\thinspace m\ge1}{k_{1}+\dots+k_{m}=n}} 1 = \sum_{n \geq 1} \mathcal{C}_{n}^{\left\{J\right\} }z^n. 
\end{equation*}
Comparing coefficients yields \eqref{eq:comp}.
\end{proof}

A direction for future research is to study the asymptotics of $x_n = \sum x_{j_i}$ which would allow us to obtain asymptotics for $\mathcal{C}_{n}^{\left\{J\right\} }$. The asymptotics of $x_n = \sum g_i x_{n-i}$ are complicated, but we have a much easier problem since the coefficients are either zero or one. Finding the asymptotics of $x_n = \sum g_i x_{n-i}$ reduces to studying the roots of the characteristic polynomial $x^n - \sum {x^{j_i}}$.

\subsection{Sum of digits}
For a given $n,$ the set of $2^{n-1}$ integers
$
\left\{ \vert \pi \vert\right\} _{\pi\in\mathcal{C}\left(n\right)}
$
coincides with the set
$
\left\{ 1+s_{2}\left(k\right)\right\} _{0\le k\le2^{n-1}-1},
$
where $s_{2}\left(k\right)$ is the number of $1'$s in the binary
expansion of $k.$ 

For example, for $n=3,$
\[
\left\{ \vert \pi \vert\right\} _{\pi\in\mathcal{C}\left(3\right)}=\left\{ 1,2,2,3\right\} 
\]
while
\[
s_{2}\left(0\right)=0,\thinspace\thinspace s_{2}\left(1\right)=1,\thinspace\thinspace s_{2}\left(2\right)=1,\thinspace\thinspace s_{2}\left(3\right)=2.
\]

This remark, together with Theorem \ref{non linear}, can be used to derive some interesting finite sums that involve the sequence $s_{2}\left(k\right).$
For example, the choice
\[
f\left(z\right)=\log\left(1+z\right),\thinspace\thinspace g\left(z\right)=\frac{z}{1-z}
\]
gives
\[
f\left(g\left(z\right)\right)=-\log\left(1-z\right)
\]
so that 
\[
x_{n}=\frac{1}{n}=\sum_{k=0}^{2^{n-1}-1}\frac{\left(-1\right)^{s_{2}\left(k\right)}}{s_{2}\left(k\right)+1},\thinspace\thinspace n\ge1
\]
This extends naturally as follows:
\begin{thm}
For $f\left(z\right) = \sum_{k \ge 1} f_k z^k,$ we have
\[
\sum_{k=0}^{2^{n-1}-1} f_{s_2\left(k\right)+1} = \left[z^n\right] f\left(\frac{z}{1-z}\right) = \sum_{k=1}^{n} f_k \binom{n-1}{n-k}.
\]
\end{thm}
This formula simply expresses the fact that for $0 \le k \le 2^{n-1}-1,$ the sequence $\left\{s_2\left(k\right)\right\}$ takes on the value zero $\binom{n-1}{1}$ times, the value one $\binom{n-2}{2}$ times, and so on. 

\section{A general formula for composition of functions}\label{sec4}
We conclude this study with a general formula for composition of functions.
Introduce the notation 
\[
\pi\models n
\]
for $\pi\in\mathcal{C}\left(n\right).$ This mirrors the notation $\lambda \vdash  n$, which states that $\lambda$ is a partition of $n$. From Thm. \ref{thm:main}, we have the formula
\[
\left(f\circ g\right)_{n}=\sum_{\pi\models n}f_{\vert \pi \vert}g_{\pi}.
\]
Now we look at 
\[
\left(f\circ\left(g\circ h\right)\right)_{n}=\sum_{\pi\models n}f_{\vert \pi \vert}\left(g\circ h\right)_{\pi}
\]
where 
\begin{align*}
\left(g\circ h\right)_{\pi} & =\left(g\circ h\right)_{\pi_{1}}\dots\left(g\circ h\right)_{\pi_{p}}\\
 & =\left(\sum_{\mu_{1}\models\pi_{1}}g_{l\left(\mu_{1}\right)}h_{\mu_{1}}\right)\dots\left(\sum_{\mu_{p}\models\pi_{p}}g_{l\left(\mu_{p}\right)}h_{\mu_{p}}\right)\\
 & =\sum_{\mu_{1}\models\pi_{1}}\dots\sum_{\mu_{p}\models\pi_{p}}g_{l\left(\mu_{1}\right)}\dots g_{l\left(\mu_{p}\right)}h_{\mu_{1}}\dots h_{\mu_{p}}.
\end{align*}
We introduce the new notation for a ``composition of compositions''
\[
\sum_{\mu\models\pi}x_{\mu}=\sum_{\mu_{1}\models\pi_{1}}\dots\sum_{\mu_{p}\models\pi_{p}}x_{\mu_{1}}\dots x_{\mu_{p}}
\]
so that we can write 
\[
\left(f\circ\left(g\circ h\right)\right)_{n}=\sum_{\mu\models\pi\models n}f_{\vert \pi \vert}g_{\vert \mu \vert}h_{\mu}
\]
%which extends naturally to 
%\[
%\left(f^{\left(1\right)}\circ\dots\circ f^{\left(k\right)}\right)_{n}=\sum_{\pi^{\left(k-1\right)}\models\dots\models\pi^{\left(1\right)}\models n}f_{l\left(\pi^{\left(1\right)}\right)}^{\left(1\right)}\dots f_{l\left(\pi^{\left(k-1\right)}\right)}^{\left(k-1\right)}f_{\pi^{(1)}}^{\left(k\right)}.
%\]
Moreover, the associativity of function compositions translates into
another equivalent expression,
\begin{align*}
%\left(f\circ\left(g\circ h\right)\right)_{n} & =\sum_{\pi\models\mu\models n}f_{\vert \mu \vert}g_{\vert \pi \vert}h_{\pi}\\
\left(\left(f\circ g\right)\circ h\right)_{n} & =\sum_{\substack{\pi \models n \\ \mu\models \vert \pi \vert}}f_{\vert \mu \vert}g_{\mu}h_{\pi}.
\end{align*}
For $n=4$ we have the 5 equivalent possibilities
\begin{eqnarray*}
f^{(1)}\circ f^{(2)}\circ f^{(3)}\circ f^{\left(4\right)} & =f^{\left(1\right)}\circ\left(f^{\left(2\right)}\circ\left(f^{\left(3\right)}\circ f^{\left(4\right)}\right)\right)=\sum_{\substack{\pi\models n\\
\mu\models\pi\\
\lambda\models\mu
}
}f_{\vert \pi \vert}^{\left(1\right)}f_{\vert \mu \vert}^{\left(2\right)}f_{\vert \lambda \vert}^{\left(3\right)}f_{\lambda}^{\left(4\right)}\\
& =\left(\left(f^{\left(1\right)}\circ f^{\left(2\right)}\right)\circ f^{\left(3\right)}\right)\circ f^{\left(4\right)}=\sum_{\substack{\pi\models n\\
\mu\models \vert \pi \vert\\
\lambda\models \vert \mu \vert
}
}f_{\vert \lambda \vert}^{\left(1\right)}f_{\lambda}^{\left(2\right)}f_{\mu}^{\left(3\right)}f_{\pi}^{\left(4\right)}\\
 & =\left(f^{\left(1\right)}\circ f^{\left(2\right)}\right)\circ\left(f^{\left(3\right)}\circ f^{\left(4\right)}\right)=\sum_{\substack{\pi\models n\\
\mu\models \vert \pi \vert\\
\lambda\models\pi
}
}f_{\vert \mu \vert}^{\left(1\right)}f_{\mu}^{\left(2\right)}f_{\vert \lambda \vert}^{\left(3\right)}f_{\lambda}^{\left(4\right)}\\
 & =f^{\left(1\right)}\circ\left(\left(f^{\left(2\right)}\circ f^{\left(3\right)}\right)\circ f^{\left(4\right)}\right)=\sum_{\substack{\pi\models n\\
\mu\models\pi\\
\lambda\models \vert \mu \vert
}
}f_{\vert \pi \vert}^{\left(1\right)}f_{\vert \lambda \vert}^{\left(2\right)}f_{\lambda}^{\left(3\right)}f_{\mu}^{\left(4\right)}\\
 & =\left(f^{\left(1\right)}\circ\left(f^{\left(2\right)}\circ f^{\left(3\right)}\right)\right)\circ f^{\left(4\right)}=\sum_{\substack{\pi\models n\\
\mu\models \vert \pi \vert\\
\lambda\models\mu
}
}f_{\vert \mu \vert}^{\left(1\right)}f_{\vert \lambda \vert}^{\left(2\right)}f_{\lambda}^{\left(3\right)}f_{\pi}^{\left(4\right)}.
\end{eqnarray*}
%Looking for a general expression, we introduce the additional notation
%for an arbitrary composition $\pi$
%\[
%\pi_{0}=\pi,\thinspace\thinspace\pi_{1}=\vert \pi \vert
%\]
with the corresponding tree representations 

\vspace{1cm}
\noindent\begin{minipage}[t]{1\columnwidth}%
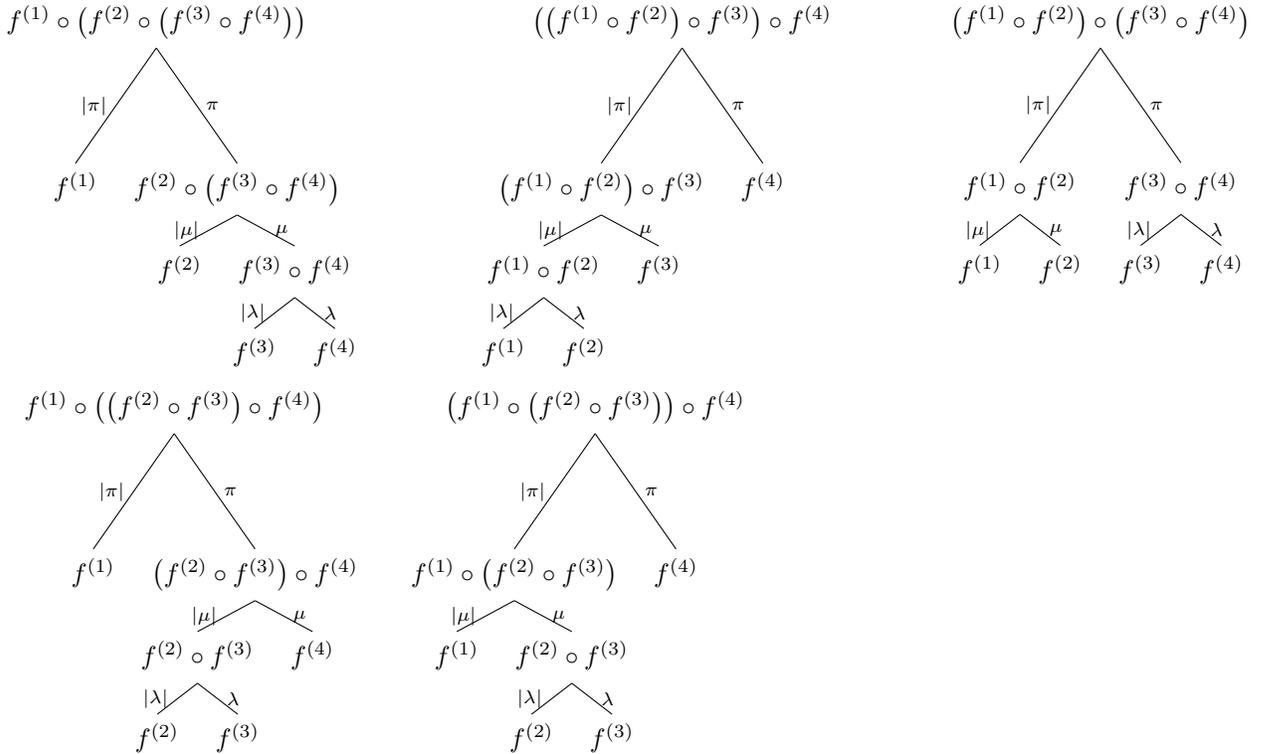
\begin{minipage}[t]{0.4\columnwidth}%
%\begin{center} 
\begin{forest}calign=fixed edge angles, 
[$f^{\left(1\right)}  \circ  \left( f^{\left(2\right)} \circ \left( f^{\left(3\right)} \circ f^{\left(4\right)} \right) \right)$[$f^{\left(1\right)}$,  edge label={node[midway,left,font=\scriptsize]{$\vert \pi \vert$}}][$f^{\left(2\right)}  \circ  \left( f^{\left(3\right)} \circ f^{\left(4\right)} \right)$,  edge label={node[midway,right,font=\scriptsize]{$\pi$}}[$f^{\left(2\right)}$,  edge label={node[midway,left,font=\scriptsize]{$\vert \mu \vert$}}][$f^{\left(3\right)} \circ f^{\left(4\right)}$,  edge label={node[midway,right,font=\scriptsize]{$\mu$}}[$f^{\left(3\right)}$,  edge label={node[midway,left,font=\scriptsize]{$\vert \lambda \vert$}}][$f^{\left(4\right)}$,  edge label={node[midway,right,font=\scriptsize]{$\lambda$}}]]]]
\end{forest} %\par\end{center}%
\end{minipage}
\begin{minipage}[t]{0.4\columnwidth}%
%\begin{center} 
\begin{forest}calign=fixed edge angles, [$\left(\left(f^{\left(1\right)}\circ f^{\left(2\right)}\right)\circ f^{\left(3\right)}\right)\circ f^{\left(4\right)}$[$\left(f^{\left(1\right)}\circ f^{\left(2\right)}\right)\circ f^{\left(3\right)}$,  edge label={node[midway,left,font=\scriptsize]{$\vert \pi \vert$}}[$f^{\left(1\right)}\circ f^{\left(2\right)}$,  edge label={node[midway,left,font=\scriptsize]{$\vert \mu \vert$}}[$f^{\left(1\right)}$,  edge label={node[midway,left,font=\scriptsize]{$\vert \lambda \vert$}}][$f^{\left(2\right)}$,  edge label={node[midway,right,font=\scriptsize]{$\lambda$}}]][$f^{\left(3\right)}$,  edge label={node[midway,right,font=\scriptsize]{$\mu$}}]][$f^{\left(4\right)}$,  edge label={node[midway,right,font=\scriptsize]{$\pi$}}]] \end{forest} 
%\par\end{center}%
\end{minipage}%
\begin{minipage}[t]{0.33\columnwidth}%
%\begin{center} 
\begin{forest}calign=fixed edge angles, 
[$\left(f^{\left(1\right)}  \circ   f^{\left(2\right)} \right) \circ \left( f^{\left(3\right)} \circ f^{\left(4\right)}  \right)$[$f^{\left(1\right)} \circ f^{\left(2\right)}$,  edge label={node[midway,left,font=\scriptsize]{$\vert \pi \vert$}}[$f^{\left(1\right)}$,  edge label={node[midway,left,font=\scriptsize]{$\vert \mu \vert$}}][$f^{\left(2\right)}$,  edge label={node[midway,right,font=\scriptsize]{$\mu$}}]][$ f^{\left(3\right)} \circ f^{\left(4\right)} $,  edge label={node[midway,right,font=\scriptsize]{$\pi$}}[$f^{\left(3\right)}$,  edge label={node[midway,left,font=\scriptsize]{$\vert \lambda \vert$}}][$f^{\left(4\right)}$,  edge label={node[midway,right,font=\scriptsize]{$\lambda$}}]]]
\end{forest} %\par\end{center}%
\end{minipage}
\end{minipage}
\vspace{1cm}
\noindent\begin{minipage}[t]{1\columnwidth}%
\begin{minipage}[l]{0.33\columnwidth}%
\begin{center} 
\begin{forest}calign=fixed edge angles, 
[$f^{\left(1\right)}\circ \left(\left(f^{\left(2\right)}\circ f^{\left(3\right)}\right)\circ f^{\left(4\right)}\right)$
[$f^{\left(1\right)}$,  edge label={node[midway,left,font=\scriptsize]{$\vert \pi \vert$}}]
[$\left(f^{\left(2\right)}\circ f^{\left(3\right)}\right)\circ f^{\left(4\right)}$,  edge label={node[midway,right,font=\scriptsize]{$\pi$}}[$f^{\left(2\right)}\circ f^{\left(3\right)}$,  edge label={node[midway,left,font=\scriptsize]{$\vert \mu \vert$}}[$ f^{\left(2\right)}$,  edge label={node[midway,left,font=\scriptsize]{$\vert \lambda \vert$}}][$ f^{\left(3\right)}$,  edge label={node[midway,right,font=\scriptsize]{$\lambda$}}]][$ f^{\left(4\right)}$,  edge label={node[midway,right,font=\scriptsize]{$\mu$}}]]]
\end{forest} \par\end{center}%
\end{minipage}%
\begin{minipage}[r]{0.33\columnwidth}%
\begin{center} \begin{forest}calign=fixed edge angles, 
[$\left(f^{\left(1\right)}\circ \left(f^{\left(2\right)}\circ f^{\left(3\right)}\right)\right)\circ f^{\left(4\right)}$
[$f^{\left(1\right)}\circ \left(f^{\left(2\right)}\circ f^{\left(3\right)}\right)$,  edge label={node[midway,left,font=\scriptsize]{$\vert \pi \vert$}}[$f^{\left(1\right)}$,  edge label={node[midway,left,font=\scriptsize]{$\vert \mu \vert$}}][$f^{\left(2\right)}\circ f^{\left(3\right)}$,  edge label={node[midway,right,font=\scriptsize]{$\mu$}}[$f^{\left(2\right)}$,edge label={node[midway,left,font=\scriptsize]{$\vert \lambda \vert$}}][$f^{\left(3\right)}$,edge label={node[midway,right,font=\scriptsize]{$\lambda$}}]]][$f^{\left(4\right)}$,  edge label={node[midway,right,font=\scriptsize]{$\pi$}}]] 
\end{forest} \par\end{center}%
\end{minipage}
\end{minipage}

Consider now the general case $f^{(1)}\circ\cdots\circ f^{(n)}$.
There are $C_{n+1}$ ways to compose these functions associatively,
where $C_{n}$ is the $n-$th Catalan number
\[
C_{n}=\frac{1}{2n+1}\binom{2n}{n}.
\]
The associativity of function composition gives us then $C_{n+1}-1$ equivalent representations for generalized composition sums.
%, since all representations are equivalent.
 
While writing down each of these $C_{n+1}$ representations is very messy, we present a simple algorithmic approach for doing so. The general formula for the composition $f^{\left(1\right)} \circ \dots f^{\left(n\right)}$ is  
\[
\sum_{  \substack{  \pi_1 \models n \\ \pi_2 \models \pi_1^{\pm} \\ \dots \\ \pi_{n-1} \models \pi_{n-2}^{\pm}  }   }
f_{\pi_{i_1}^{\pm}}^{\left(1\right)}
f_{\pi_{i_2}^{\pm}}^{\left(2\right)}
\dots
f_{\pi_{i_n}^{\pm}}^{\left(n\right)},
\]
with the notation $\pi_{i}^{+}=\pi_{i}$ and $\pi_{i}^{-}=\vert \pi_{i} \vert$. We draw the tree representation for function composition and determine the summation set by the following rules:
\begin{itemize}
\item
If the parent of the level corresponding to $\pi_{i}$ is a right child, we sum over $\pi_{i-1}^+ = \pi_{i-1}$. If it is a left child we sum over $\pi_{i-1}^- = \vert \pi_{i-1} \vert$
\item 
Select $f^{(i)}$ and look where it is a leaf. If it is a right child, we sum over the factor $f^{(i)}_{\pi_j}$, where $j$ is the depth of the leaf $f^{(i)}$. If $f^{(i)}$ is a leaf and left child we sum over the factor $f^{(i)}_{\vert \pi_j \vert}$.
\item
At the top level, the previous rules do not apply and we always sum over $f^{(i)}_{\pi_1}$ and $\pi_1 \models n$.
\end{itemize}
This way, the subscript is determined by the \textit{leaf's} position while the summation set is determined by the \textit{parent's} position in the corresponding tree. For instance, select the following tree: 

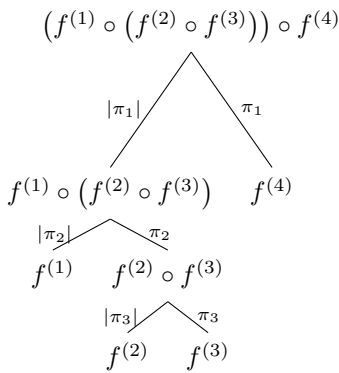
\begin{center} \begin{forest}calign=fixed edge angles, 
[$\left(f^{\left(1\right)}\circ \left(f^{\left(2\right)}\circ f^{\left(3\right)}\right)\right)\circ f^{\left(4\right)}$
[$f^{\left(1\right)}\circ \left(f^{\left(2\right)}\circ f^{\left(3\right)}\right)$,  edge label={node[midway,left,font=\scriptsize]{$\vert \pi_1 \vert$}}[$f^{\left(1\right)}$,  edge label={node[midway,left,font=\scriptsize]{$\vert \pi_2 \vert$}}][$f^{\left(2\right)}\circ f^{\left(3\right)}$,  edge label={node[midway,right,font=\scriptsize]{$\pi_2$}}[$f^{\left(2\right)}$,edge label={node[midway,left,font=\scriptsize]{$\vert \pi_3 \vert$}}][$f^{\left(3\right)}$,edge label={node[midway,right,font=\scriptsize]{$\pi_3$}}]]][$f^{\left(4\right)}$,  edge label={node[midway,right,font=\scriptsize]{$\pi_1$}}]].
\end{forest} \par\end{center}%
Consider $f^{(1)}$. It's a left child with depth $2$, so it has the subscript $\vert \pi_2 \vert$. Its parent $f^{\left(1\right)}\circ \left(f^{\left(2\right)}\circ f^{\left(3\right)}\right)$ is also a left child so the summation includes the terms $\pi_2 \models \vert \pi_1 \vert$. $f^{(3)}$ is a right child with depth $3$ so we sum over $f^{(3)}_{\pi_3}$. Its parent is a right child so the summation goes over $\pi_3 \models \pi_2$. Repeating this procedure for $f^{(2)}$ and $f^{(4)}$ allows us to recover our previous expression
\begin{equation}
\left(f^{\left(1\right)}\circ\left(f^{\left(2\right)}\circ f^{\left(3\right)}\right)\right)\circ f^{\left(4\right)}=\sum_{\substack{\pi_1\models n\\
\pi_2\models \vert \pi_1 \vert\\
\pi_3\models\pi_2
}
}f_{\vert \pi_2 \vert}^{\left(1\right)}f_{\vert \pi_3 \vert}^{\left(2\right)}f_{\pi_3}^{\left(3\right)}f_{\pi_1}^{\left(4\right)}.
\end{equation}

\section{Conclusion}
In this paper, we recognized Fuchs' generalized PI tree as a graphical method to represent sums over compositions. Then, the introduction of the Fa\'{a} di Bruno formula allowed us to introduce  further set of weights based on the number of parts in every compositions. Trees are then an easy 'bookkeeping' method to visualize the classical Fa\'{a} di Bruno formula. The introduction of this Fa\'{a} di Bruno formula also allowed us to synthesize many past works, since the nonlinear weights $f_n$ generalize most of the existing literature on compositions. Through this generating function methodology,
% for the first time
we were also able to find a generating function for sums over all parts of all compositions of $n$. Together, our results unite the existing literature on composition sums and provide an efficient method to graphically visualize them. For the first time, we also study iterated sums over compositions and link them to tree representations.

\end{document}